\documentclass[11pt,a4paper]{article}
\usepackage{amsmath}
\usepackage{amssymb}
\usepackage{mathrsfs}
\usepackage{hyperref}
\numberwithin{equation}{section}

\evensidemargin=\oddsidemargin\textwidth=142mm \textwidth=160mm
\newtheorem{theorem}{Theorem}[section]
\newtheorem{lemma}{Lemma}[section]

\newtheorem{corollary}{Corollary}[section]
\newtheorem{remark}{Remark}[section]
\newenvironment{proof}[1][Proof]{\begin{trivlist}
\item[\hskip \labelsep {\bfseries #1}]}{\end{trivlist}}

\begin{document}
\title {Generalized Hilbert Operator Acting on Weighted Bergman Spaces and  on Dirichlet Spaces\footnote{ The research was supported by the National Natural Science Foundation of China (Grant No. 11671357)}}
\author{Shanli Ye\footnote{Corresponding author. E-mail address:  slye@zust.edu.cn}
\quad\quad Guanghao Feng\footnote{E-mail address:gh945917454@foxmail.com}     \\
(\small \it School of Sciences, Zhejiang University of Science and Technology, Hangzhou
310023, China$)$}

 \date{}
\maketitle
\begin{abstract}

Let $\mu$ be a positive Borel measure on the interval [0,1). For $\beta > 0$, The generalized Hankel matrix
$\mathcal{H}_{\mu,\beta}= (\mu_{n,k,\beta})_{n,k\geq0}$ with entries $\mu_{n,k,\beta}= \int_{[0.1)}\frac{\Gamma(n+\beta)}{n!\Gamma(\beta)} t^{n+k}d\mu(t)$,  induces formally the operator
$$\mathcal{H}_{\mu,\beta}(f)(z)=\sum_{n=0}^\infty \left(\sum_{k=0}^\infty \mu_{n,k,\beta}a_k\right)z^n$$
on the space of all analytic function $f(z)=\sum_{k=0}^ \infty a_k z^n$   in the unit disc $\mathbb{D}$. In this paper, we characterize those positive Borel  measures on $[0,1)$ such that
$\mathcal{H}_{\mu,\beta}(f)(z)= \int_{[0,1)} \frac{f(t)}{{(1-tz)^\beta}} d\mu(t)$ for all in weighted Bergman Spaces $A_{\alpha}^p(0<p<\infty,\; \alpha>-1)$, and among them we describe those for which $\mathcal{H}_{\mu,\beta}(\beta>0)$ is a bounded(resp.,compact) operator on weighted Bergman spaces and Dirichlet spaces.
\\
{\small\bf Keywords}\quad {Generalized Hilbert operator, Weighted Bergman Spaces, Dirichlet spaces, Carleson measure
 \\
    {\small\bf 2020 MR Subject Classification }\quad 47B35, 30H20, 30H30\\}
\end{abstract}
\maketitle

\section{Introduction}\label{s1}

Let $\mu$ be a positive Borel measure on the interval [0,1). for $\beta > 0$, we define $\mathcal{H}_{\mu,\beta}= (\mu_{n,k,\beta})_{n,k\geq0}$ to be generalized Hankel matrix with entries $\mu_{n,k,\beta}= \int_{[0.1)}\frac{\Gamma(n+\beta)}{n!\Gamma(\beta)} t^{n+k}d\mu(t).$

For analytic functions $f(z)=\sum_{k=0}^ \infty a_k z^n$ , the generalized Hilbert operator define as
\begin{align}\label{eq1.1}
\mathcal{H}_{\mu,\beta}(f)(z)=\sum_{n=0}^\infty \left(\sum_{k=0}^\infty \mu_{n,k,\beta}a_k\right)z^n,
\end{align}
whenever the right hand side makes sense and defines an analytic function in $\mathbb{D}$.

In recent decades, the generalized Hilbert operator $\mathcal{H}_{\mu,1}$ which induced by the Hankel matrix $\mathcal{H}_\mu$ have been studied extensively. For example, Galanopoulos and Pel\'{a}ez\cite{GHM} characterized the Borel measure $\mu$ for which the Hankel operator is a bounded (resp.,compact) operator on $H^1$. Then Chatzifountas, Girela and Pel\'{a}ez \cite{CGH} extended this work the all Hardy spaces $H^p$ with $0<p< \infty$. In \cite{GGHO}, Girela and Merch\'{a}n studied the operator acting on certain conformally invariant spaces of analytic functions on the disk. For the case $ \beta = 2$, it had been studied the operator in \cite{YDBL,YDBE} and the $\mathcal{H}_{\mu,2}$ was called the Derivative-Hilbert operator. In \cite{YGH}, the $\mathcal{H}_{\mu,\beta}$ was called the generalized Hilbert operator.

Another generalized Hilbert-integral operator related to $\mathcal{H}_{\mu,\beta}$ denoted by
$\mathcal{I}_{\mu,\beta} (\beta\in R^+$) is defined by
$$ \mathcal{I}_{\mu,\beta}(f)(z)=\int_{[0,1)} \frac{f(t)}{{(1-tz)^\beta}} d\mu(t),$$
whenever the right hand side makes sense and defines an analytic function in $\mathbb{D}$.

In this paper, we consider the operators
\begin{align}
   & \mathcal{H}_{\mu,\beta}: A_\alpha^p\rightarrow A_{\beta-2}^q,  \quad 0<p\leq q <\infty \; and \; q\geq1, \; \alpha>-1 \;and \;\beta>1.\notag \\
   & \mathcal{H}_{\mu,\beta}: \mathcal{D}_\alpha^p\rightarrow \mathcal{D}_{\beta-1}^q, \quad  1<p\leq q <\infty, \alpha>-1 \;and \;\beta>0.\notag\notag
\end{align}
The aim is to study the boundedness (resp.,compactness) of $\mathcal{H}_{\mu,\beta}$.

In this article we characterize the positive Borel measures $\mu$ for which the operator which $ \mathcal{I}_{\mu,\beta}$ and $\mathcal{H}_{\mu,\beta}$ is well defined in the weighted Bergman space
 $A_\alpha^p$. Then we give the necessary and sufficient conditions  such that operator $\mathcal{H}_{\mu,\beta}$ is bounded (resp.,compact) from the weighted Bergman spaces $A_\alpha^p(0<p<\infty, \; \alpha>-1)$ into the space $A_{\beta-2}^q (\beta>1,\;q \geq 1$ and $p\leq q <\infty )$, and we also give necessary and sufficient conditions such that operator $\mathcal{H}_{\mu,\beta}$ is bounded (resp.,compact) from the Dirichlet type space $\mathcal{D}_\alpha^p(\alpha>p-1,\; 0<p<\infty)$ into the space $\mathcal{D}_{\beta-1}^q(\beta>q,\; p\leq q<\infty$ and $q > 1)$, and get a series of corollaries.

\section{Notation and Preliminaries}\label{s2}

Let $\mathbb{D}$ denote the open unit disk of the complex plane, and let $H(\mathbb{D})$ denote the set of all analytic functions in $\mathbb{D}$.

For $0 < p < \infty $ and $\alpha > -1$ the weighted Bergman space $A_\alpha^p$ consists of those
$f \in H(\mathbb{D})$ such that
$$\|f\|_{A_\alpha^p}\overset{def}{=} \left((\alpha+1)\int_\mathbb{D}|f(z)|^p (1-|z|^2)^\alpha dA(z) \right)^{1/p}< \infty.$$
Here, $dA$ stands for the area measure on $\mathbb{D}$, normalized so that the total area of $\mathbb{D}$ is 1. Thus $dA(z) = \frac{1}{\pi}dxdy = \frac{1}{\pi}rdrd\theta$. The unweighted Bergman space $A_0^p$ is simply denoted by $A^p$. We refer to \cite{DBS,ZOT} for the notation and results about Bergman spaces.

The space of Dirichlet type $\mathcal{D}_\alpha^p$ ($0 < p <\infty $ and $\alpha > -1$) consists of those $f \in H(\mathbb{D})$ such that $f' \in A_\alpha^p$. In other words, a function $f \in H(\mathbb{D})$ belongs to $\mathcal{D}_\alpha^p$ if and only if
$$\|f\|_{\mathcal{D}_\alpha^p}\overset{def}{=}|f(0)|+\left((\alpha+1)\int_\mathbb{D}|f'(z)|^p (1-|z|^2)^\alpha dA(z) \right)^{1/p}< \infty.$$
Finally, we recall that a function $f\in H(\mathbb{D})$ is said to be a Bloch function if
$$\|f\|_\mathcal{B}\overset{def}{=} |f(0)|+\sup_{z\in \mathbb{D}}(1-|z|^2)|f'(z)|< \infty.$$
The space of all Bloch functions is denoted by $\mathcal{B}$. A classical reference for the theory of Bloch functions is \cite{GOB,POB}. The little Bloch space $\mathcal{B}_0$ is the subspace of the Bloch space consisting of those functions $f \in H(\mathbb{D})$ such that.

$$\lim_{|z|\rightarrow0}(1-|z|^2)|f'(z)|=0.$$

Let us recall the knowledge of Carleson measure, which is a very useful tool in the study of Banach spaces of analytic functions. For $0<s<\infty$, a positive Borel measure $\mu$ on $\mathbb{D}$ will be called a $s$-Carleson measure, if there exists a positive constant $C$ such that
$$\mu(S(I))\leq C|I|^s .$$
The Carleson square $S(I)$ is defined as
$$S(I)=\{z=re^{i\theta}:e^{i\theta}\in I;1-\frac{|I|}{2\pi}\leq r \leq 1\}.$$
where $I$ is an interval of $\partial \mathbb{D}$, $|I|$ denotes the length of $I$, if $\mu $ satisfies $\lim_{|I|\rightarrow 0} \frac{\mu(S(I)) }{|I|^s}=0$, we call $\mu$ is a vanishing $s$-Carleson measure.

Let $\mu$ be a positive Borel measure on $\mathbb{D}$. For $0\leq \alpha < \infty$ and $ 0<s< \infty $ we say that $\mu$ is $\alpha$-logarithmic $s$-Carleson measure, if there exists a positive constant $C$ such that
 $$\frac{\mu(S(I))(\log\frac{ 2\pi }{|I|})^\alpha}{|I|^s} \leq C \quad  \quad I \subset \partial \mathbb{D}.$$
If $\mu(S(I))(\log\frac{ 2\pi }{|I|})^\alpha=o(|I|^s)$, as  $|I|\rightarrow 0$, we say that $\mu$ is vanishing $\alpha$-logarithmic $s$-Carleson measure\cite{MVL,ZOL}.

Suppose that $\mu$ is a $s$-Carleson measure on $\mathbb{D}$, we see that the identity mapping $i$ is well defined from $A_\alpha^p$ into $L^q(\mathbb{D},\mu).$ Let $\mathcal{N}(\mu)$ be the norm of $i$. For $0<r<1$, let
$$d\mu_r(z)=\chi_{r<|z|<1}(t)d\mu(t).$$
Then $\mu$ is a vanishing $s$-Carleson measure if and only if
\begin{align}\label{eq2.1}
  \mathcal{N}(\mu_r) \rightarrow 0 \quad as  \quad r\rightarrow 1^-.
\end{align}

A positive Borel measure on $[0,1)$ also can be seen as a Borel measure on $\mathbb{D}$ by identifying it with the measure $\mu$ defined by
$$\tilde{\mu}(E)=\mu(E\bigcap [0,1)).$$
for any Borel subset E of $\mathbb{D}$. Then a positive Borel measure $\mu$ on $[0,1)$ can be seen as an $s$-Carleson measure on $\mathbb{D}$, if
$$\mu([t,1)) \leq C(1-t)^s,  \quad 0 \leq t<1 .$$
Also, we have similar statements for vanishing $s$-Carleson measures, $\alpha$-logarithmic
$s$-Carleson and vanishing $\alpha$-logarithmic $s$-Carleson measures.

As usual, throughout this paper, $C$ denotes a positive constant which depends only on the displayed parameters but not necessarily the same from one occurrence to the next. For any given $p > 1$, $p'$ will denote the conjugate index of $p$, that is, $1/p + 1/p' = 1$.
\section{Conditions such that $\mathcal{H}_{\mu,\beta}$ is well defined on Weighted Bergman spaces}\label{s3}
In this section, we find the sufficient condition such that $\mathcal{H}_{\mu,\beta}(\beta>0)$ are well defined in $A^p_\alpha(0<p< \infty,\; \alpha>-1)$, and obtain that $\mathcal{H}_{\mu,\beta}(f)=\mathcal{I}_{\mu_,\beta}(f)$, for all $f\in A^p_\alpha$, with the certain condition.

We first recall a result about the coefficients of functions in weighted Bergman spaces.
\begin{lemma}\cite[~p.~191]{JTC} \label{le3.1}
Let $f \in H(\mathbb{D})$, $f(z)=\sum_{n=1}^\infty a_n z^n$ and let $-1< \alpha < \infty$.
\begin{itemize}
		\item[$(\romannumeral1)$] If $0<p\leq 1$, and $f \in A_\alpha^p$, then $a_n=o(n^{(\alpha+2)/p-1})$.
		\item[$(\romannumeral2)$] If $1\leq p<\infty$ and $f \in A_\alpha^p$, then $a_n=o(n^{(\alpha+1)/p})$.
        \item[$(\romannumeral3)$] If $0<p\leq 2$, then
		$$\sum_{n=1}^{\infty}n^{-\alpha-1}|a_n|^p<\infty\Rightarrow f\in A_\alpha^p\Rightarrow \sum_{n=1}^{\infty}n^{p-\alpha-3}|a_n|^p<\infty.$$
        \item[$(\romannumeral4)$] If $2\leq p<\infty$, then
		$$\sum_{n=1}^{\infty}n^{p-\alpha-3}|a_n|^p<\infty\Rightarrow f\in A_\alpha^p \Rightarrow
		\sum_{n=1}^{\infty}n^{-\alpha-1}|a_n|^p<\infty.$$
	\end{itemize}
\end{lemma}

\begin{lemma}\label{le3.2}
For $0<p\leq q<\infty$ and $\alpha>-1$, $\mu$ is a $\left(\frac{(2+\alpha) q}{p}\right)$-Carleson measure if and only if there exists a positive constant $C$ such that
	\begin{align}\label{eq3.1}	
		\left\{\int_{\mathbb{D}}|f(z)|^qd\mu(z)\right\}^{1/q}\leq C\|f\|_{A_\alpha^p},\quad for~all~f\in A_\alpha^p.
	\end{align}
\end{lemma}
The result was proved by several authors which can be found in \cite{HCM} and \cite{LFR},

\begin{theorem}\label{Th3.1}
Suppose that $0 < p < \infty$ and $-1<\alpha<\infty $ and let $\mu$ be a positive Borel measure on $[0,1)$. Then the power series in (\ref{eq1.1}) defines a well defined analytic function in $\mathbb{D}$ for every $f\in A_\alpha^p$ in any of the three following cases.
	\begin{itemize}
		\item[$(\romannumeral1)$] $\mu$ is a $(\frac{\alpha+2}{p})$-Carleson measure if $0<p\leq 1$.
		\item[$(\romannumeral2)$] $\mu$ is a $\left(\frac{\alpha+2-(p-1)^2}{p}\right)$-Carleson measure if $1\leq p\leq 2$.
		\item[$(\romannumeral3)$] $\mu$ is a $(\frac{\alpha+1}{p})$-Carleson measure if $2\leq p<\infty$.
	\end{itemize}
Furthermore, in such cases we have that
\begin{align}\label{eq3.2} \mathcal{H}_{\mu,\beta}(f)(z)=\int_{[0,1)}\frac{f(t)}{(1-tz)^{\beta}}d\mu(t)
=\mathcal{I}_{\mu,\beta}(f)(z).
\end{align}
\end{theorem}

\begin{proof}
(\romannumeral1) Suppose that $0<p \leq 1$ and $\mu$ is a $(\frac{\alpha+2}{p})$-Carleson measure. Then (\ref{eq3.1}) gives
$$\int_{[0,1)}|f(t)|d\mu(t)\leq C\|f\|_{A_\alpha^p},\quad for~all~f\in A_\alpha^p.$$

Fix $f(z) = \sum _{k=0}^ \infty a_k z^k \in A_\alpha^p$ and $z$ with $|z|<r$, $0 <r< 1$. It follows that

\begin{align}
\int_{[0,1)} \frac{|f(t)|}{{|1-tz|^\beta}} d\mu(t) & \leq \frac{1}{(1-r)^\beta} \int_{[0,1)}|f(t)|d\mu (t)    \notag \\
& \leq C \frac{1}{(1-r)^\beta}||f||_{A_\alpha^p} .  \notag
\end{align}

This implies that the integral $\int_{[0,1)} \frac{f(t)}{{(1-tz)^\beta}} d\mu(t)$ uniformly converges, and that
\begin{align}\label{eq3.3}
 \mathcal{I}_{\mu_,\beta}(f)(z)= &\int_{[0,1)} \frac{f(t)}{{(1-tz)^\beta}} d\mu(t) \notag \\
  =& \sum_{n=0}^\infty\frac{\Gamma(n+\beta)}{n!\Gamma(\beta)}\left(\int_{[0,1)}t^n f(t)d\mu (t)\right)z^n.
\end{align}

Since $\mu$ is $(\frac{\alpha+2}{p})$-Carleson measure. Take $f(z) = \sum _{k=0}^ \infty a_k z^k \in A_\alpha^p$, By \cite[Proposition 1]{CGH}, and Lemma \ref{le3.1},  we have that there exists $C>0$ such that
$$|\mu_{n,k,\beta}|=\frac{\Gamma(n+\beta)}{n!\Gamma(\beta)}|\mu_{n+k}|\leq \frac{C}{(k+1)^\frac{\alpha+2}{p}}\frac{\Gamma(n+\beta)}{n!\Gamma(\beta)},$$
$$|a_k| \leq C(k+1)^{(\alpha+2)/p-1} \quad for \; all \; n,k.$$
Then it follows that, for every $n$,
\begin{align}
  \sum_{k=0}^\infty |\mu_{n,k,\beta}||a_k| & \leq C\frac{\Gamma(n+\beta)}{n!\Gamma(\beta)} \sum_{k=0}^\infty \frac{|a_k|}{(k+1)^{(\alpha+2)/p}} =C\frac{\Gamma(n+\beta)}{n!\Gamma(\beta)} \sum_{k=0}^\infty \frac{|a_k|^p |a_k|^{(1-p)}}{(k+1)^{(\alpha+2)/p}}\notag \\
   & \leq C\frac{\Gamma(n+\beta)}{n!\Gamma(\beta)})\sum_{k=0}^\infty \frac{|a_k|^p (k+1)^{(\frac{\alpha+2}{p}-1)(1-p)}}{(k+1)^{(\alpha+2)/p}} \notag\\
   &=C\frac{\Gamma(n+\beta)}{n!\Gamma(\beta)}\sum_{k=0}^\infty(k+1)^{p-\alpha-3}|a_k|^p \notag.
\end{align}
This, Lemma \ref{le3.1} and (\ref{eq3.3}) readily yield that the series in (\ref{eq1.1}) is well defined for all $z\in \mathbb{D}$ and that

\begin{align}
 \mathcal{H}_{\mu,\beta}(f)(z)= \int_{[0,1)} \frac{f(t)}{{(1-tz)^\beta}} d\mu(t)\quad z \in \mathbb{D},  \notag
\end{align}
this give that $\mathcal{H}_{\mu,\beta}(f)=\mathcal{I}_{\mu,\beta} (f)$.

(\romannumeral2) When $1\leq p \leq 2$, since $\mu$ is a $\left(\frac{\alpha+2-(p-1)^2}{p}\right)$-Carleson, Arguing as in the preceding one, we obtain that there exists a constant $C>0$ such that, for every $n$,
$$\sum_{k=0}^\infty |\mu_{n,k,\beta}||a_k| \leq C\frac{\Gamma(n+\beta)}{n!\Gamma(\beta)}\sum_{k=0}^\infty(k+1)^{p-\alpha-3}|a_k|^p.$$

Then Lemma \ref{le3.1} again implies that the power series in (\ref{eq1.1}) is well defined and
$$\sum_{k=0}^\infty \mu_{n,k,\beta}a_k =\frac{\Gamma(n+\beta)}{n!\Gamma(\beta)}\int_{[0,1)}t^n f(t)d\mu (t). $$

It follows that for each $z\in \mathbb{D}$,
\begin{align}
&\sum_{n=0}^\infty \frac{\Gamma(n+\beta)}{n!\Gamma(\beta)}\left(\int_{[0,1)}t^n f(t)d\mu (t)\right)|z|^n \notag\\
=&\sum_{n=0}^\infty \frac{\Gamma(n+\beta)}{n!\Gamma(\beta)}\left(\sum_{k=0}^\infty \mu_{n,k}|a_k|\right)|z|^n \notag \\
\leq & C\sum_{n=0}^\infty \frac{\Gamma(n+\beta)}{n!\Gamma(\beta)}|z|^n\sum_{k=0}^\infty(k+1)^{p-\alpha-3}|a_k|^p\notag \\
 \leq &\frac{C}{(1-|z|)^\beta}\sum_{k=0}^\infty(k+1)^{p-\alpha-3}|a_k|^p<\infty, \notag
\end{align}
we deduce that
\begin{align}\label{eq3.4}
 \mathcal{H}_{\mu_,\beta}(f)(z)= & \sum_{n=0}^\infty\frac{\Gamma(n+\beta)}{n!\Gamma(\beta)}\left(\int_{[0,1)}t^n f(t)d\mu (t)\right)z^n \notag \\
 =& \int_{[0,1)}  \sum_{n=0}^\infty\frac{\Gamma(n+\beta)}{n!\Gamma(\beta)}(tz)^n f(t) d\mu(t) \notag \\
 =& \int_{[0,1)} \frac{f(t)}{{(1-tz)^\beta}} d\mu(t),
\end{align}

which together implies that $\mathcal{H}_{\mu,\beta}$ is a well defined for all $z\in \mathbb{D}$, and $\mathcal{H}_{\mu,\beta}(f)=\mathcal{I}_{\mu_,\beta} (f)$.

(\romannumeral3)The proof is similar to that of (\romannumeral2). We omit the details.
\end{proof}

\begin{remark} Theorem \ref{Th3.1} is the promotion of Theorem 2.1 in \cite{YDBE}, when $\alpha=0$, Theorem \ref{Th3.1} is same as Theorem 2.1 in \cite{YDBE}.
\end{remark}
\section{  Bounededness of  $\mathcal{H}_{\mu,\beta}$ acting on Weighted Bergman spaces and Dirichlet spaces} \label{s4}

In this section, we mainly characterize those measures $\mu$ for which $\mathcal{H}_{\mu,\beta}$ is a bounded operator on the weighted Bergman spaces $A_\alpha^p$ and the Dirichlet spaces $\mathcal{D}_\alpha^p$ for some different $\beta$, $p$ and $\alpha$.

\begin{lemma}\label{Le4.1}\cite{GHMA}
If $1 < p < \infty$ and $\alpha > -1$, then the dual of $A_\alpha^p$ can be identified with $A_\gamma^{p'}$ and $\gamma$ is any number with $\gamma > -1$, under the pairing
$$<h,f>_{A_{p,\alpha,\gamma}}=\int_{\mathbb{D}}h(z)\overline{f(z)}(1-|z|^2)^{\frac{\alpha}{p}+\frac{\gamma}{p'}} dA(z), \quad h\in A_\alpha^p, \quad f\in A_\gamma^{p'}.$$

If $1 < p < \infty $ and $\alpha > -1$, then the dual of $\mathcal{D}_\alpha^p$ can be identified with $\mathcal{D}_\gamma^{p'}$ where and $\gamma$ is any number with $\gamma > -1$, under the pairing
$$<h,f>_{\mathcal{D}_{p,\alpha,\gamma}}=h(0)\overline{f(0)}+\int_{\mathbb{D}}h'(z)\overline{f'(z)}(1-|z|^2)^{\frac{\alpha}{p}+\frac{\gamma}{p'}} dA(z), \quad h\in \mathcal{D}_\alpha^p, \quad f\in \mathcal{D}_\gamma^{p'}.$$

\end{lemma}

\begin{theorem}\label{Th4.01}
Suppose that $0<p <\infty $, $\alpha>-1$ . Let $\mu$ be a positive Borel measure on $[0,1)$ which satisfies the condition in Theorem \ref{Th3.1}.
\begin{itemize}
\item [($\romannumeral1$)] If  $q\geq p$ and $q > 1$, then $\mathcal{H}_{\mu,\alpha+2}$ is a bounded operator from $A_\alpha^p$ into $A_\alpha^q$ if and only if  $\mu$ is a $\left(\frac{\alpha+2}{p}+\frac{\alpha+2}{q'}\right)$-Carleson measure.
\item [($\romannumeral2$)] If  $q\geq p$ and $q = 1$, then $\mathcal{H}_{\mu,\alpha+2}$ is a bounded operator from $A_\alpha^p$ into $A_{\alpha}^1$ if and only if $\mu$ is a 1-logarithmic $(\frac{\alpha+2}{p})$-Carleson measure.
\end{itemize}
\end{theorem}

The Theorem \ref{Th4.01} is also closely related to Theorem \ref{Th4.1}, it we suffices to prove Theorem \ref{Th4.1}
\begin{theorem}\label{Th4.1}
Suppose that $0<p <\infty $, $\alpha>-1$ and $\beta>1$. Let $\mu$ be a positive Borel measure on $[0,1)$ which satisfies the condition in Theorem \ref{Th3.1}.
\begin{itemize}
\item [($\romannumeral1$)] If  $q\geq p$ and $q > 1$, if $\mathcal{H}_{\mu,\beta}$ is a bounded operator from $A_\alpha^p$ into $A_{\beta-2}^q$ then $\mu$ is a $\left(\frac{\alpha+2}{p}+\frac{\beta}{q'}\right)$-Carleson measure, conversely if $\mu$ is a $\left(\max\{\alpha+2,\beta\}(\frac{1}{p}+\frac{1}{q'})\right)$-Carleson measure then $\mathcal{H}_{\mu,\beta}$ is a bounded operator from $A_\alpha^p$ into $A_{(\beta-2)}^q$.
\item [($\romannumeral2$)] If  $q\geq p$ and $q = 1$, then $\mathcal{H}_{\mu,\beta}$ is a bounded operator from $A_\alpha^p$ into $A_{\beta-2}^1$ if and only if $\mu$ is a 1-logarithmic $(\frac{\alpha+2}{p})$-Carleson measure.
\end{itemize}
\end{theorem}

\begin{proof}
Suppose that $0 <p< \infty $ , $\alpha>-1$ and $\beta>1$. Since $\mu$ satisfies the condition in Theorem \ref{Th3.1}, as in the proof of Theorem \ref{Th3.1}, we obtain that
\begin{align*}
	\int_{[0,1)}|f(t)|d\mu(t)< \infty,\quad f\in A_\alpha^p.
\end{align*}
Hence, it follows that
\begin{align}\label{eq4.1}
		\begin{split}
			& (\alpha+1)\int_{\mathbb{D}} \int_{[0,1)}\left|\frac{f(t) g( z)(1-|z|^2)^\alpha}{(1- t z)^{\beta}}\right| d \mu(t) dA(z) \\
			\leq & \frac{1}{(1-r)^{\beta}} \int_{[0,1)}|f(t)|d\mu(t)\int_{\mathbb{D}}(\alpha+1)(1-|z|^2)^\alpha|g( z)| dA(z) \\
			\leq & \frac{C}{(1-r)^{\beta}}\|g\|_{A_\alpha^1},~0\leq r<1,~f\in A_\alpha^p,~g\in A_\alpha^1,
		\end{split}
\end{align}
we can obtain that the reproducing kernel function of $A_\alpha^p$ is
$$K_\alpha(z,w)=\frac{1}{(1-z\bar{w})^{2+\alpha}}.$$
Now the reproducing property
$$f(z)=\int_{\mathbb{D}}(\alpha+1)f(w)(1-|w^2|)^\alpha K_\alpha(z,w)dA(w), \quad for \; all\; f \in A_\alpha^1. $$
Thus by using Fubini's theorem imply that
\begin{align}\label{eq4.2}
\begin{split}
		&\int_{\mathbb{D}} \overline{\mathcal{H}_{\mu,\beta}(f)(z)} g( z)(1-|z|^2)^{\beta-2} d A(z)\\
		=&\int_{\mathbb{D}} \int_{[0,1)} \frac{\overline{f(t)} d \mu(t)}{(1- t \bar{z})^{\beta}} g( z)(1-|z|^2)^{\beta-2} d A(z)\\
		=&\int_{[0,1)} \int_{\mathbb{D}} \frac{(1-|z|^2)^{\beta-2}g( z)}{(1- t \bar{z})^{\beta}} d A(z) \overline{f(t)} d \mu(t) \\
		=& \frac{1}{\beta-1}\int_{[0,1)} \overline{f(t)} g(t) d \mu(t),\quad f\in A_\alpha^p ,\;g\in A_{\beta-2}^1.
	\end{split}
\end{align}

(\romannumeral1)First consider the case $0<p<\infty$, $q\geq 1$ and $q\geq p$. Using (\ref{eq4.2}) and the duality theorem for Lemma \ref{Le4.1}, take $(A_{\beta-2}^q)^*\cong A_{\beta-2}^{q'}$ and $(A_{\beta-2}^{q'})^*\cong A_{\beta-2}^q(q>1)$. We obtain that
$\mathcal{H}_{\mu,\beta}$ is a bounded operator from $A_\alpha^p$ into $A_{\beta-2}^q$ if and only if there exists a positive constant $C$ such that
\begin{align}\label{eq4.3}
 \left| \int_{[0,1)}\overline{f(t)}g(t) d\mu(t) \right| \leq C\|f\|_{A_\alpha^p} \|g\|_{A_{\beta-2}^{q'}}, \quad  f \in A_\alpha^p, \; g\in A_{\beta-2}^{q'}.
\end{align}

Assume that $\mathcal{H}_{\mu,\beta}$ is a bounded operator from $A_\alpha^p$ into $A_{\beta-2}^{q'}$. Take the families of text functions
$$f_{a}(z)=\left(\frac{1-a^{2}}{(1-a z)^{2}}\right)^{(\alpha+2)/p} \quad \text{and} \quad g_{a}(z)=\left(\frac{1-a^{2}}{(1-a z)^{2}}\right)^{\beta/q'},\quad 0<a<1.$$
A calculation shows that $\{f_a\} \subset A_\alpha^p$, $\{g_a\} \subset A_{\beta-2}^{q'}$ and
\begin{align}
\sup_{a \in [0,1)}||f||_{A_\alpha^p} < \infty \quad and \quad \sup_{a \in [0,1)}||g||_{A_{\beta-2}^{q'}} < \infty . \notag
\end{align}
It follows that
\begin{align}
  \infty & > C \;\sup_{a \in [0,1)}\|f\|_{A_\alpha^p}\sup_{a \in [0,1)}\|g\|_{A_{\beta-2}^{q'}}  \notag \\
   & \geq \left|\int_{[0,1)}f_a(t)g_a(t)d\mu(t) \right| \notag \\
   & \geq  \int_{[a,1)} \left(\frac{1-a^2}{(1-at)^2} \right)^{(\alpha+2)/p} \left(\frac{1-a^2}{(1-at)^2}\right)^{\beta/q'}   d\mu(t)\notag \\
   & \geq \int_{[a,1)} \left(\frac{1-a^2}{(1-at)^2} \right)^{(\alpha+2)/p+\beta/q'}d\mu(t) \notag \\
    & \geq \frac{1}{(1-a^2)^{(\alpha+2)/p+\beta/q'}} \mu([a,1)).
\end{align}
This is equivalent to saying that $\mu$ is a $\left(\frac{\alpha+2}{p}+\frac{\beta}{q'}\right)$-Carleson measure.

On the other hand, if $\mu$ is a $\left(\max\{\alpha+2,\beta\}(\frac{1}{p}+\frac{1}{q'})\right)$-Carleson measure. Let $s=1+p/q'$. Then $s'=1+q'/p$. By Lemma \ref{le3.2}, we obtain that there exists a constant $C>0$ such that
\begin{align}
	\left(\int_{[0,1)}\left|f(t)\right|^s d\mu(t)\right)^{1/s}\leq C\|f\|_{A_\alpha^p},\quad \text{for~all~}f\in A_{\alpha}^p,
\end{align}
and
\begin{align}		
	\left(\int_{[0,1)}\left|g(t)\right|^{s'} d\mu(t)\right)^{1/s'}\leq C\|g\|_{A_{\beta-2}^{q'}},\quad \text{for~all~}g\in A_{\beta-2}^{q'}.
\end{align}
Then by H\"{o}lder's inequality,
\begin{align*}
	\int_{[0,1)} \left|f(t)\right|\left| g(t)\right| d \mu(t)&\leq \left(\int_{[0,1)}\left|f(t)\right|^s d\mu(t)\right)^{1/s}\left(\int_{[0,1)}\left|g(t)\right|^{s'} d\mu(t)\right)^{1/s'}\\
&\leq C\|f\|_{A_{\alpha}^p}\|g\|_{A_{\beta-2}^{q'}}.
\end{align*}

Hence, (\ref{eq4.3}) holds and then $\mathcal{H}_{\mu,\beta}$ is a bounded operator from $A_\alpha^p$ into $A_{\beta-2}^q$.

(\romannumeral2) Let us first recall the duality theorem for $A_\alpha^1$ \cite{ZOT}, we see that $A_\alpha^1$ can be identified as the dual of the little Bloch space under the pairing
\begin{align}\label{eq4.7}
  <h,g> =\int_{\mathbb{D}}h(z)\overline{g(z)}(1-|z|^2)^{\alpha} dA(z), \quad h\in \mathcal{B}_0, \quad g\in A_\alpha^{1}.
\end{align}
we obtain that
$\mathcal{H}_{\mu,\beta}$ is a bounded operator from $A_\alpha^p$ into $A_{\beta-2}^1$ if and only if there exists a positive constant $C$ such that
\begin{align}\label{eq4.8}
 \left| \int_{[0,1)}\overline{f(t)}g(t) d\mu(t) \right| \leq C\|f\|_{A_\alpha^p} \|g\|_{\mathcal{B}}, \quad  f \in A_\alpha^p, \; g\in \mathcal{B}_0.
\end{align}

Assume that $\mathcal{H}_{\mu,\beta}$ is a bounded operator from $A_\alpha^p$ into $A_{\beta-2}^{1}$. Take the families of text functions
$$f_{a}(z)=\left(\frac{1-a^{2}}{(1-a z)^{2}}\right)^{(\alpha+2)/p} \quad \text{and} \quad g_{a}(z)=\log \frac{2}{1-az},\quad 0<a<1.$$
A calculation shows that $\{f_a\} \subset A_\alpha^p$, $\{g_a\} \subset \mathcal{B}_0$ and
\begin{align}
\sup_{a \in [0,1)}||f||_{A_\alpha^p} < \infty \quad and \quad \sup_{a \in [0,1)}||g||_{\mathcal{B}} < \infty . \notag
\end{align}
It follows that
\begin{align}
  \infty & > C \;\sup_{a \in [0,1)}||f||_{A_\alpha^p}\sup_{a \in [0,1)}||g||_{\mathcal{B}}  \notag \\
   & \geq \left|\int_{[0,1)}f_a(t)g_a(t)d\mu(t) \right| \notag \\
   & \geq  \int_{[a,1)} \left(\frac{1-a^2}{(1-at)^2} \right)^{(\alpha+2)/p} \log \frac{2}{1-at}   d\mu(t)\notag \\
    & \geq \frac{\log \frac{2}{1-a} }{(1-a)^{(\alpha+2)/p}} \mu([a,1)).
\end{align}
This is equivalent to saying that $\mu$ is a 1-logarithmic $(\frac{\alpha+2}{p})$-Carleson measure.

Conversely, we assume that $\mu$ is a 1-logarithmic $(\frac{\alpha+2}{p})$-Carleson measure. It is well known that any function $g \in \mathcal{B}$ \cite{POB} has the property
$$|g(z)|\leq C\|g\|_{\mathcal{B}}\log\frac{2}{1-|z|}.$$

Thus, if $0 <p \leq 1$, for every $f \in A_\alpha^p$ and $g \in \mathcal{B}_0 \subset \mathcal{B}$, we have
\begin{align}
\left| \int_{[0,1)}\overline{f(t)}g(t) d\mu(t) \right| &\leq C \|g\|_{\mathcal{B}} \int_{[0,1)}|f(t)|\log\frac{2}{1-t}d\mu(t). \notag
\end{align}

Using \cite[Proposition 2.5]{CCO}, we obtain that $\log \frac{2}{1-t}d\mu(t)$ is a $(\frac{\alpha+2}{p})$-Carleson measure. Hence, Lemma \ref{le3.2} implies that
\begin{align}
\left| \int_{[0,1)}\overline{f(t)}g(t) d\mu(t) \right| &\leq C \|f\|_\alpha^p \|g\|_{\mathcal{B}} \quad  f \in A_\alpha^p, \; g\in \mathcal{B}_0.
\end{align}
Therefore, $\mathcal{H}_{\mu,\beta}$ is a bounded operator from $A_\alpha^p$ into $A_{\beta-2}^1$.
\end{proof}

\begin{corollary}
Suppose that $1\leq p <\infty $, $\alpha>-1$. Let $\mu$ be a positive Borel measure on $[0,1)$ which satisfies the condition in Theorem \ref{Th3.1}.
\begin{itemize}
\item [($\romannumeral1$)] If  $p > 1$, then $\mathcal{H}_{\mu,\alpha+2}$ is a bounded operator on $A_\alpha^p$  if and only if $\mu$ is a $(\alpha+2)$-Carleson measure.
\item [($\romannumeral2$)] If $p = 1$, then $\mathcal{H}_{\mu,\alpha+2}$ is a bounded operator on $A_\alpha^1$  if and only if $\mu$ is a 1-logarithmic $(\alpha+2)$-Carleson measure.
\end{itemize}
\end{corollary}

Now, it is well known that $A_\alpha^p = D_{\alpha+p}^p$ (see \cite[Theorem 4.28]{ZOT}). Hence, regarding Dirichlet spaces Theorem \ref{Th4.1} says the following.

\begin{corollary}
Suppose that $0<p <\infty $, $\alpha>p-1$ and $\beta>q+1$. Let $\mu$ be a positive Borel measure on $[0,1)$ which satisfies the condition in Theorem \ref{Th3.1}.
\begin{itemize}
\item [($\romannumeral1$)] If  $q\geq p$ and $q > 1$, if $\mathcal{H}_{\mu,\beta-q}$ is a bounded operator from $\mathcal{D}_\alpha^p$ into $\mathcal{D}_{\beta-2}^q$ then $\mu$ is a $\left(\frac{\alpha+2}{p}+\frac{\beta-q}{q'}-1\right)$-Carleson measure.
     Conversely if $\mu$ is a $\left(\max\{\alpha+2-p,\beta-q\}(\frac{1}{p}+\frac{1}{q'})\right)$-Carleson measure then $\mathcal{H}_{\mu,\beta-q}$ is a bounded operator from $\mathcal{D}_\alpha^p$ into $\mathcal{D}_{\beta-2}^q$.
\item [($\romannumeral2$)] If  $q\geq p$ and $q = 1$, then $\mathcal{H}_{\mu,\beta-1}$ is a bounded operator from $\mathcal{D}_\alpha^p$ into $\mathcal{D}_{\beta-2}^1$ if and only if $\mu$ is a 1-logarithmic $(\frac{\alpha+2}{p}-1)$-Carleson measure.
\end{itemize}
\end{corollary}

\begin{corollary}
Suppose that $1\leq p <\infty $, $\alpha>p-1$. Let $\mu$ be a positive Borel measure on $[0,1)$ which satisfies the condition in Theorem \ref{Th3.1}.
\begin{itemize}
\item [($\romannumeral1$)] If  $p > 1$, then $\mathcal{H}_{\mu,\alpha+2-p}$ is a bounded operator on $\mathcal{D}_\alpha^p$ if and only if $\mu$ is a $(\alpha+2-p)$-Carleson measure.
\item [($\romannumeral2$)] If $p = 1$, then $\mathcal{H}_{\mu,\alpha+2-p}$ is a bounded operator on $\mathcal{D}_\alpha^1$  if and only if $\mu$ is a 1-logarithmic $(\alpha+1)$-Carleson measure.
\end{itemize}
\end{corollary}

Then we discuss those measures $\mu$ for which $\mathcal{H}_{\mu,\beta}$ is a bounded operator from $\mathcal{D}_\alpha^p$ into $\mathcal{D}_{\beta-1}^q$.

\begin{theorem}\label{Th4.03}
Suppose that $q > 1$ and  $0<p\leq q <\infty$, $\alpha>p-1$. Let $\mu$ be a positive Borel measure on $[0,1)$ which satisfies the condition in Theorem \ref{Th3.1}. Then $\mathcal{H}_{\mu,\alpha+1-p}$ is a bounded operator from $\mathcal{D}_\alpha^p$ into $\mathcal{D}_{\alpha-p}^q$ if and only if $\mu$ is a $\left(\frac{\alpha+2-p}{p}+\frac{\alpha+2-p}{q'}\right)$-Carleson measure.
\end{theorem}

The Theorem \ref{Th4.03} is also closely related to Theorem \ref{Th4.3}, thus it suffices to prove Theorem \ref{Th4.3}
\begin{theorem}\label{Th4.3}
Suppose that $q > 1$ and  $0<p\leq q <\infty$, $\alpha>p-1$ and $\beta>q$. Let $\mu$ be a positive Borel measure on $[0,1)$ which satisfies the condition in Theorem \ref{Th3.1}.
\begin{itemize}
\item [($\romannumeral1$)] If $\mathcal{H}_{\mu,\beta}$ is a bounded operator from $\mathcal{D}_\alpha^p$ into $\mathcal{D}_{\beta-1}^q$ then $\mu$ is a $\left(\frac{\alpha+2}{p}+\frac{\beta+1}{q'}-1\right)$-Carleson measure.
\item [($\romannumeral2$)] If $\mu$ is a $\left(\max\{\alpha+2-p,\beta+1\}(\frac{1}{p}+\frac{1}{q'})\right)$-Carleson measure then $\mathcal{H}_{\mu,\beta}$ is a bounded operator from $\mathcal{D}_\alpha^p$ into $\mathcal{D}_{\beta-1}^q$.
\end{itemize}
\end{theorem}
\begin{proof}
By using Fubini's theorem and  reproducing property imply that
\begin{align}\label{eq4.11}
\begin{split}
		&\int_{\mathbb{D}} \overline{\mathcal{H}_{\mu,\beta}(f)'(z)} g'( z)(1-|z|^2)^{\beta-1} d A(z)\\
		=&\int_{\mathbb{D}} \int_{[0,1)}t\beta \frac{\overline{f(t)} d \mu(t)}{(1- t \bar{z})^{\beta+1}} g'( z)(1-|z|^2)^{\beta-1} d A(z)\\
		=&\int_{[0,1)} \int_{\mathbb{D}} t\beta \frac{(1-|z|^2)^{\beta-1}g'( z)}{(1- t \bar{z})^{\beta+1}} d A(z) \overline{f(t)} d \mu(t) \\
		=& \frac{\beta}{\beta-1}\int_{[0,1)} t\overline{f(t)} g'(t) d \mu(t),\quad f\in \mathcal{D}_\alpha^p ,\;g'\in A_{\beta-1}^1.
	\end{split}
\end{align}

(\romannumeral1)First consider the case $0<p<\infty$, $q\geq 1$ and $q\geq p$. Using (\ref{eq4.11}) and the duality theorem for Lemma \ref{Le4.1}, take $(\mathcal{D}_{\beta-1}^q)^*\cong \mathcal{D}_{\beta-1}^{q'}$ and $(\mathcal{D}_{\beta-1}^{q'})^*\cong \mathcal{D}_{\beta-1}^q(q>1)$. we obtain that
$\mathcal{H}_{\mu,\beta}$ is a bounded operator from $\mathcal{D}_\alpha^p$ into $\mathcal{D}_{\beta-1}^q$ if and only if there exists a positive constant $C$ such that
\begin{align}\label{eq4.12}
 \left| \int_{[0,1)}t\overline{f(t)}g'(t) d\mu(t) \right| \leq C\|f\|_{\mathcal{D}_\alpha^p} \|g\|_{\mathcal{D}_{\beta-1}^{q'}}, \quad  f \in \mathcal{D}_\alpha^p, \; g'\in A_{\beta-1}^{q'}.
\end{align}

Assume that $\mathcal{H}_{\mu,\beta}$ is a bounded operator from $\mathcal{D}_\alpha^p$ into $\mathcal{D}_{\beta-1}^q$. Take the families of text functions
$$f_{a}(z)= \frac{1}{a}\frac{(1-a^{2})^{(\alpha+2)/p}}{(1-a z)^{2{(\alpha+2)/p}-1}} \quad \text{and} \quad g'_{a}(z)=\left(\frac{1-a^{2}}{(1-a z)^{2}}\right)^{(\beta+1)/q'},\quad 0<a<1.$$
A calculation shows that $\{f_a\} \subset \mathcal{D}_\alpha^p$, $\{g_a\} \subset \mathcal{D}_{\beta-1}^{q'}$ and
\begin{align}
\sup_{a \in [0,1)}||f||_{\mathcal{D}_\alpha^p} < \infty \quad and \quad \sup_{a \in [0,1)}||g||_{\mathcal{D}_{\beta-1}^{q'}} < \infty . \notag
\end{align}
It follows that
\begin{align}
  \infty & > C \;\sup_{a \in [0,1)}||f||_{\mathcal{D}_\alpha^p}\sup_{a \in [0,1)}||g||_{\mathcal{D}_{\beta-1}^{q'}}  \notag \\
   & \geq \left|\int_{[0,1)}tf_a(t)g_a'(t)d\mu(t) \right| \notag \\
   & \geq  \int_{[a,1)} \frac{t}{a}\frac{(1-a^{2})^{(\alpha+2)/p}}{(1-a t)^{2{(\alpha+2)/p}-1}} \left(\frac{1-a^2}{(1-at)^2}\right)^{(\beta+1)/q'}   d\mu(t)\notag \\
    & \geq \frac{1}{(1-a^2)^{(\alpha+2)/p+(\beta+1)/q'-1}} \mu([a,1)).
\end{align}
This is equivalent to saying that $\mu$ is a $\left(\frac{\alpha+2}{p}+\frac{\beta+1}{q'}-1\right)$-Carleson measure.

(\romannumeral2) On the other hand, if $\mu$ is a $\left(\max\{\alpha+2-p,\beta+1\}(\frac{1}{p}+\frac{1}{q'})\right)$-Carleson measure. Let $s=1+p/q'$. Then $s'=1+q'/p$. We find $A_\alpha^p = \mathcal{D}_{\alpha+p}^p$ which means $f \in A_{\alpha-p}^p$. By Lemma \ref{le3.2}, we obtain that there exists a constant $C>0$ such that
\begin{align}
	\left(\int_{[0,1)}\left|f(t)\right|^s d\mu(t)\right)^{1/s}\leq C\|f\|_{A_{(\alpha-p)}^p}=C\|f\|_{\mathcal{D}_{\alpha}^p},\quad \text{for~all~}f\in \mathcal{D}_{\alpha}^p,
\end{align}
and
\begin{align}		
	\left(\int_{[0,1)}\left|g'(t)\right|^{s'} d\mu(t)\right)^{1/s'}\leq C\|g'\|_{A_{\beta-1}^{q'}} \leq C\|g\|_{\mathcal{D}_{\beta-1}^{q'}},\quad \text{for~all~}g'\in A_{\beta-1}^{q'}.
\end{align}
Then by H\"{o}lder's inequality,
\begin{align*}
	\int_{[0,1)} \left|f(t)\right|\left| g'(t)\right| d \mu(t)&\leq \left(\int_{[0,1)}\left|f(t)\right|^s d\mu(t)\right)^{1/s}\left(\int_{[0,1)}\left|g'(t)\right|^{s'} d\mu(t)\right)^{1/s'}\\
&\leq C\|f\|_{\mathcal{D}_{\alpha}^p}\|g\|_{\mathcal{D}_{\beta-1}^{q'}}.
\end{align*}
Hence, (\ref{eq4.12}) holds and then $\mathcal{H}_{\mu,\beta}$ is a bounded operator from $\mathcal{D}_\alpha^p$ into $\mathcal{D}_{\beta-1}^q.$
\end{proof}

\begin{corollary}
Suppose that $1 < p < \infty $, $\alpha>p-1$. Let $\mu$ be a positive Borel measure on $[0,1)$ which satisfies the condition in Theorem \ref{Th3.1}. Then $\mathcal{H}_{\mu,\alpha+1-p}$ is a bounded operator from $\mathcal{D}_\alpha^p$ into $\mathcal{D}_{\alpha-p}^p$ if and only if $\mu$ is a $(\alpha+2-p)$-Carleson measure.
\end{corollary}

\begin{corollary}
Suppose that $q > 1$ and $0<p\leq q <\infty$, $\alpha>-1$ and $\beta>0$. Let $\mu$ be a positive Borel measure on $[0,1)$ which satisfies the condition in Theorem \ref{Th3.1}.
\begin{itemize}
\item [($\romannumeral1$)] If $\mathcal{H}_{\mu,\beta+q}$ is a bounded operator from $A_\alpha^p$ into $A_{\beta-1}^q$ then $\mu$ is a $\left(\frac{\alpha+2}{p}+\frac{\beta+1+q}{q'}\right)$-Carleson measure.
\item [($\romannumeral2$)] If $\mu$ is a $\left(\max\{\alpha+2,\beta+1+q\}(\frac{1}{p}+\frac{1}{q'})\right)$-Carleson measure then $\mathcal{H}_{\mu,\beta+q}$ is a bounded operator from $A_\alpha^p$ into $A_{\beta-1}^q$.
\end{itemize}
\end{corollary}

\begin{corollary}
Suppose that $1 < p < \infty $, $\alpha>p-1$. Let $\mu$ be a positive Borel measure on $[0,1)$ which satisfies the condition in Theorem \ref{Th3.1}. Then $\mathcal{H}_{\mu,\alpha+1}$ is a bounded operator from $A_\alpha^p$ into $A_{\alpha-p}^p$ if and only if $\mu$ is a $(\alpha+2)$-Carleson measure.
\end{corollary}

\section{ Compactness of  $\mathcal{H}_{\mu,\beta}$ acting on Weighted Bergman spaces and Dirichlet spaces}\label{s5}
In this section we characterize the compactness of the
Generalized Hilbert $\mathcal{H}_{\mu,\beta}$. We begin with the following Lemma \ref{Le5.1} which is useful to deal with the compactness.

\begin{lemma}\label{Le5.1}
For $0 < p < \infty $ and $0 < q < \infty$, $\alpha>-1$ and $\beta>1$. Suppose that $\mathcal{H}_{\mu,\beta}$ is a bounded operator from $A_\alpha^p$ into $A_{\beta-2}^q$. Then $\mathcal{H}_{\mu,\beta}$ is a compact operator if and only if for any bounded sequence $\{f_n\}$ in $A_\alpha^p$ which converges uniformly to 0 on every compact subset of  $\mathbb{D}$, we have $\mathcal{H}_{\mu,\beta}(f_n) \rightarrow 0$ in $A_{\beta-2}^q$, for Dirichlet spaces also have an analogous conclusion.
\end{lemma}

The proof is similar to that of \cite[Proposition 3.11]{CCO}. We omit the details.

\begin{theorem}\label{Th5.01}
Suppose that $0<p <\infty $, $\alpha>-1$ . Let $\mu$ be a positive Borel measure on $[0,1)$ which satisfies the condition in Theorem \ref{Th3.1}.
\begin{itemize}
\item [($\romannumeral1$)] If  $q\geq p$ and $q > 1$, then $\mathcal{H}_{\mu,\alpha+2}$ is a compact operator from $A_\alpha^p$ into $A_\alpha^q$ if and only if  $\mu$ is a vanishing $\left(\frac{\alpha+2}{p}+\frac{\alpha+2}{q'}\right)$-Carleson measure.
\item [($\romannumeral2$)] If  $q\geq p$ and $q = 1$, then $\mathcal{H}_{\mu,\alpha+2}$ is a compact operator from $A_\alpha^p$ into $A_{\alpha}^1$ if and only if $\mu$ is a vanishing 1-logarithmic $(\frac{\alpha+2}{p})$-Carleson measure.
\end{itemize}
\end{theorem}

 The Theorem \ref{Th5.01} is closely related to Theorem \ref{Th5.1}, thus it suffices to prove Theorem \ref{Th5.1}

\begin{theorem}\label{Th5.1}
Suppose that $0<p <\infty $, $\alpha>-1$ and $\beta>1$. Let $\mu$ be a positive Borel measure on $[0,1)$ which satisfies the condition in Theorem \ref{Th3.1}.
\begin{itemize}
\item [($\romannumeral1$)] If  $q\geq p$ and $q > 1$, if $\mathcal{H}_{\mu,\beta}$ is a compact operator from $A_\alpha^p$ into $A_{\beta-2}^q$ then $\mu$ is a vanishing $\left(\frac{\alpha+2}{p}+\frac{\beta}{q'}\right)$-Carleson measure, conversely if $\mu$ is a vanishing $\left(\max\{\alpha+2,\beta\}(\frac{1}{p}+\frac{1}{q'})\right)$-Carleson measure then $\mathcal{H}_{\mu,\beta}$ is a compact operator from $A_\alpha^p$ into $A_{\beta-2}^q$.
\item [($\romannumeral2$)] If  $q\geq p$ and $q = 1$, then $\mathcal{H}_{\mu,\beta}$ is a compact operator from $A_\alpha^p$ into $A_{\beta-2}^1$ if and only if $\mu$ is a vanishing 1-logarithmic $(\frac{\alpha+2}{p})$-Carleson measure.
\end{itemize}
\end{theorem}

\begin{proof}
(\romannumeral1) First consider  $q\geq p$ and $q > 1$. Suppose that $\mathcal{H}_{\mu,\beta}$ is a compact operator from  $A_\alpha^p$ into $A_{\beta-2}^q$. Let $\{a_n\}\subset (0,1)$ be any sequence with $a_n\rightarrow 1$. We set
$$f_{a_n}(z)=\left(\frac{1-a_n^2}{(1-a_n z)^2} \right)^{(\alpha+2)/p}, \quad  z\in \mathbb{D}.$$
Then $f_{a_n}(z) \in A_\alpha^p$, $\sup_{n \geq 1} ||f_{a_n}||_{H^p} < \infty $ and $f_{a_n}\rightarrow0 $, uniformly on any compact subset of $\mathbb{D}$. Using Lemma \ref{Le5.1} and bearing in mind that $\mathcal{H}_{\mu,\beta}$ is a compact operator from $A_\alpha^p$ into $A_{\beta-2}^q$, we obtain that $\{\mathcal{H}_{\mu,\beta}(f_{a_n} )\}$ converges to 0 in $A_{\beta-2}^q$. This and (\ref{eq4.2}) imply that
\begin{align}\label{eq5.1}
	\begin{split}
		&\lim_{n\rightarrow \infty}\int_{[0,1)} \overline{f_{a_n}(t)} g(t) d \mu(t)\\
		=&\lim_{n\rightarrow \infty}\int_{\mathbb{D}} \overline{\mathcal{H}_{\mu,\beta}(f_{a_n})(z)}g(z)d(1-|z|^2)^{\beta-2} A(z)=0,\quad g\in A_{\beta-2}^{q'}.
	\end{split}
\end{align}
Now we set
$$g_{a_n}(z)=\left(\frac{1-a_n^{2}}{(1-{a_n}z)^{2}}\right)^{(\beta)/q'},\quad z\in\mathbb{D}.$$
It is easy to check that $g\in A_{\beta-2}^{q'}.$

For $r\in (a_n,1),$ we deduce that
\begin{align*}
	&\int_{[0,1)} \overline{f_{a_n}(t)} g_{a_n}(t) d \mu(t) \\
	\geq& \int_{a_n}^{1}\left(\frac{1-a_n^{2}}{(1-{a_n} t)^{2}}\right)^{(\alpha+2)/p} \left(\frac{1-a_n^{2}}{(1-{a_n} t)^{2}}\right)^{\beta/q'}  d\mu(t) \\
	\geq& \frac{C}{\left(1-a_n\right)^{(\alpha+2)/p+\beta/q'}} \mu([a_n, 1)).
\end{align*}
Then (\ref{eq5.1}) and the fact that $\{a_n\}\subset (0,1)$ is a sequence with $a_n\rightarrow 1$ imply that
$$\lim_{a\rightarrow 1^-}\frac{1}{\left(1-a\right)^{(\alpha+2)/p+\beta/q'}} \mu([a, 1))=0,$$
which shows that $\mu$ is a vanishing $\left(\frac{\alpha+2}{p}+\frac{\beta}{q'}\right)$-Carleson measure.

On the other hand, suppose now that $\mu$ is a vanishing $\left(\max\{\alpha+2,\beta\}(\frac{1}{p}+\frac{1}{q'})\right)$-Carleson measure. Let $\{f_n\}_{n=1}^{\infty}$ be a sequence of $A_\alpha^p$ functions with $\sup_{n\geq 1}\|f_n\|_{A_\alpha^p}<\infty$ and such that $\{f_n\}\rightarrow 0$, uniformly on $\mathbb{D}$. Then by Lemma \ref{Le5.1}, it is enough to prove that $\mathcal{H}_{\mu,\beta}(f_n)\rightarrow 0$ in $A_{\beta-2}^q$.

Taking $g\in A_{\beta-2}^{q'}$ and $r\in[0,1)$, we have
$$\int_{[0,1)}\left|f_{n}(t)\right||g(t)| d \mu(t) =\int_{[0, r)}\left|f_{n}(t)\right||g(t)| d \mu(t)+\int_{[r, 1)}\left|f_{n}(t)\right||g(t)| d \mu(t).$$
For the first term on the right side tends to $0$ since $\{f_n\}\rightarrow 0$ uniformly on compact subsets of $\mathbb{D}$.
For the second term, arguing as in proof of the boundedness (Theorem \ref{Th4.1} (\romannumeral1)), we obtain that
\begin{align*}
	\int_{[r,1)} \left|f_n(t)\right|\left| g(t)\right| d \mu(t)&\leq \left(\int_{[0,1)}\left|f_n(t)\right|^s d\mu_r(t)\right)^{1/s}\left(\int_{[0,1)}\left|g(t)\right|^{s'} d\mu_r(t)\right)^{1/s'}\\
	&\leq \left(\mathcal{N}(\mu_r)\right)^2\|f\|_{A_\alpha^p}\|g\|_{A_{\beta-2}^{q'}}.
\end{align*}
It tends to $0$ by (\ref{eq2.1}) and then
$$
\lim _{n \rightarrow \infty} \int_{[0,1)}\left|f_{n}(t) \| g(t)\right| d \mu(t)=0, \quad \text { for all } g \in A_{\beta-2}^{q'}.
$$
Thus
$$
\lim _{n \rightarrow \infty} \left\vert \int_{\mathbb{D}} \overline{\mathcal{ H}_{\mu,\beta}(f_{n}) (z)} g\left(z\right)(1-|z|^{\beta-2}) d A(z) \right\vert=0, \quad \text { for all } g \in A_{\beta-2}^{q'}.
$$
It means $\mathcal{H}_{\mu,\beta}(f_{n})\rightarrow0\; in \; A_{\beta-2}^q$, by Lemma \ref{Le5.1}, we obtain $\mathcal{H}_{\mu,\beta}$ is a compact operator from $A_\alpha^p$ into $A_{\beta-2}^q$.

(\romannumeral2) Let $0<p\leq1$. Suppose that $\mathcal{H}_{\mu,\beta}$ is a compact operator from $A_\alpha^p$ into $A_{\beta-2}^1$. Let $\{a_n\}\subset (0,1)$ be any sequence with $a_n\rightarrow 1$ and $f_{a_n}$ defines like in  (\romannumeral1). Lemma \ref{Le5.1} implies that $\{\mathcal{H}_{\mu,\beta}(f_{a_n} )\}$ converges to 0 in $A_{\beta-1}^1$. Then we have
\begin{align}\label{eq5.2}
	\begin{split}
		&\lim_{n\rightarrow \infty}\int_{[0,1)} \overline{f_{a_n}(t)} g(t) d \mu(t)\\
		=&\lim_{n\rightarrow \infty}\int_{\mathbb{D}} \overline{\mathcal{H}_{\mu,\beta}(f_{a_n})(z)}g(z)d(1-|z|^2)^{\beta-2} A(z)=0,\quad g\in \mathcal{B}_0.
	\end{split}
\end{align}
Now we set
$$g_(z)=\log \frac{2}{1-z}.$$
It is well know that $g\in \mathcal{B}_0$. For $r\in (a_n,1)$, we deduce that
\begin{align}
  & \int_{[0,1)}\overline{f_{a_n}(t)}g(t)d\mu(t) \notag \\
  & \geq C \int_{[a_n,1)} \left(\frac{1-a_n^2}{(1-a_nt)^2} \right)^{(\alpha+2)/p} \log \frac{2}{1- t}  d\mu(t) \notag \\
  & \geq  \frac{C\log \frac{2}{1- a_n}}{(1-a_n)^{(\alpha+2)/p}} \mu([a_n,1)). \notag
\end{align}
Letting $a_n\rightarrow1^-$ as $n\rightarrow \infty$, we have
$$\lim_{a\rightarrow1^-}\frac{\log \frac{2}{1- a_n}}{(1-a_n^2)^{(\alpha+2)/p}} \mu([a_n,1))=0.$$

We can obtain that $\mu$ is a vanishing 1-logarithmic $(\frac{\alpha+2}{p})$-Carleson measure.

On the other hand, suppose that $\mu$ is a vanishing 1-logarithmic $(\frac{\alpha+2}{p})$-Carleson measure. Let $d\nu(t)=\log\frac{2}{1-t}d\mu(t)$. By Proposition 2.5 of \cite{GGHO}, we know that $\nu$ is a vanishing $(\frac{\alpha+2}{p})$-Carleson.
Let $\{f_n\}_{n=1}^\infty $ be a sequence of $A_\alpha^p$ functions with $\sup_{n\geq 1} || f_n||_{ A_\alpha^p} < \infty$, and such that $\{f_n\} \rightarrow 0$, uniformly on any compact subset of $\mathbb{D}$. Then by Lemma \ref{Le5.1}, it is enough to prove that $\{\mathcal{H}_{\mu,\beta}(f_n )\}\rightarrow0$ in $A_{\beta-2}^1$.

Then for every $g\in\mathcal{B}_0$, $0<r<1$, using (\ref{eq2.1}), we deduce that
\begin{align}
		\int_{[0,1)}\left|f_{n}(t)\right||g(t)| d \mu(t) &=\int_{[0, r)}\left|f_{n}(t)\right||g(t)| d \mu(t)+\int_{[r, 1)}\left|f_{n}(t)\right||g(t)| d \mu(t) \notag\\
		& \leq \int_{[0, r)}\left|f_{n}(t)\right||g(t)| d \mu(t)+C\|g\|_{\mathcal{B}} \int_{[r, 1)} |f_n (t)|\log\frac{2}{1-t} d \mu(t)  \notag\\
		& \leq \int_{[0, r)}\left|f_{n}(t)\right||g(t)| d \mu(t)+C\|g\|_{\mathcal{B}} \int_{[0, 1)} |f_n (t)|d \nu_r(t) \notag\\
		& \leq \int_{[0, r)}\left|f_{n}(t)\right||g(t)| d \mu(t)+ C\mathcal{N}(\nu_r)\sup_{n\geq 1}\|f_n\|_{A^p}\|g\|_{\mathcal{B}}.
\end{align}
Hence, (\ref{eq2.1}), the conditions $\{f_n\}\rightarrow0$, uniformly on compact subsets of $\mathbb{D}$, and $\sup_{n\geq 1}\|f_n\|_{A^p}$ imply that
$$
\lim _{n \rightarrow \infty} \int_{[0,1)}\left|f_{n}(t) \| g(t)\right| d \mu(t)=0, \quad \text { for all } g \in \mathcal{B}_0.
$$
Using this and (\ref{eq5.2}), we have
$$
\lim _{n \rightarrow \infty} \left\vert \int_{\mathbb{D}} \overline{\mathcal{H}_{\mu,\beta}(f_{n}) (z)} g\left(z\right)(1-|z|^2)^{\beta-2} d A(z) \right\vert=0, \quad \text { for all } g \in \mathcal{B}_0.
$$
It means $\mathcal{H}_{\mu,\beta}(f_{n})\rightarrow0\; in \; A_{\beta-2}^1$, by Lemma \ref{Le5.1} we obtain $\mathcal{H}_{\mu,\beta}$ is a compact operator from $A_\alpha^p$ into $A_{\beta-2}^1$.
\end{proof}

\begin{corollary}
Suppose that $1\leq p <\infty $, $\alpha>-1$. Let $\mu$ be a positive Borel measure on $[0,1)$ which satisfies the condition in Theorem \ref{Th3.1}.
\begin{itemize}
\item [($\romannumeral1$)] If  $p > 1$, then $\mathcal{H}_{\mu,\alpha+2}$ is a compact operator on $A_\alpha^p$  if and only if $\mu$ is a vanishing $(\alpha+2)$-Carleson measure.
\item [($\romannumeral2$)] If $p = 1$, then $\mathcal{H}_{\mu,\alpha+2}$ is a compact operator on $A_\alpha^1$  if and only if $\mu$ is a vanishing 1-logarithmic $(\alpha+2)$-Carleson measure.
\end{itemize}
\end{corollary}
It is well known that $A_\alpha^p = D_{\alpha+p}^p$ (see \cite[Theorem 4.28]{ZOT}). Hence, we can immediately obtain the following.
\begin{corollary}
Suppose that $0<p <\infty $, $\alpha>p-1$ and $\beta>q+1$. Let $\mu$ be a positive Borel measure on $[0,1)$ which satisfies the condition in Theorem \ref{Th3.1}.
\begin{itemize}
\item [($\romannumeral1$)] If  $q\geq p$ and $q > 1$, if $\mathcal{H}_{\mu,\beta-q}$ is a compact operator from $\mathcal{D}_\alpha^p$ into $\mathcal{D}_{\beta-2}^q$ then $\mu$ is a vanishing $\left(\frac{\alpha+2}{p}+\frac{\beta-q}{q'}-1\right)$-Carleson measure.\\
     Conversely, if $\mu$ is a vanishing $\left(\max\{\alpha+2-p,\beta-q\}(\frac{1}{p}+\frac{1}{q'})\right)$-Carleson measure then $\mathcal{H}_{\mu,\beta-q}$ is a compact operator from $\mathcal{D}_\alpha^p$ into $\mathcal{D}_{\beta-2}^q$.
\item [($\romannumeral2$)] If  $q\geq p$ and $q = 1$, then $\mathcal{H}_{\mu,\beta-1}$ is a compact operator from $\mathcal{D}_\alpha^p$ into $\mathcal{D}_{\beta-2}^1$ if and only if $\mu$ is a vanishing 1-logarithmic $(\frac{\alpha+2}{p}-1)$-Carleson measure.
\end{itemize}
\end{corollary}

\begin{corollary}
Suppose that $1\leq p <\infty $, $\alpha>p-1$. Let $\mu$ be a positive Borel measure on $[0,1)$ which satisfies the condition in Theorem \ref{Th3.1}.
\begin{itemize}
\item [($\romannumeral1$)] If  $p > 1$, then $\mathcal{H}_{\mu,\alpha+2-p}$ is a compact operator on $\mathcal{D}_\alpha^p$ if and only if $\mu$ is a vanishing $(\alpha+2-p)$-Carleson measure.
\item [($\romannumeral2$)] If $p = 1$, then $\mathcal{H}_{\mu,\alpha+2-p}$ is a compact operator on $\mathcal{D}_\alpha^1$  if and only if $\mu$ is a vanishing 1-logarithmic $(\alpha+1)$-Carleson measure.
\end{itemize}
\end{corollary}

Then we discuss those measures $\mu$ for which $\mathcal{H}_{\mu,\beta}$ is a compact operator from $\mathcal{D}_\alpha^p$ into $\mathcal{D}_{\beta-1}^q$.

\begin{theorem}\label{Th5.03}
Suppose that $q > 1$ and  $0<p\leq q <\infty$, $\alpha>p-1$. Let $\mu$ be a positive Borel measure on $[0,1)$ which satisfies the condition in Theorem \ref{Th3.1}. Then $\mathcal{H}_{\mu,\alpha+1-p}$ is a compact operator from $\mathcal{D}_\alpha^p$ into $\mathcal{D}_{\alpha-p}^q$ if and only if $\mu$ is a vanishing $\left(\frac{\alpha+2-p}{p}+\frac{\alpha+2-p}{q'}\right)$-Carleson measure.
\end{theorem}

The Theorem \ref{Th5.03} is also closely related to Theorem \ref{Th5.3}, thus it suffices to prove Theorem \ref{Th5.3}
\begin{theorem}\label{Th5.3}
Suppose that  $q > 1$ and $0<p\leq q <\infty$, $\alpha>p-1$ and $\beta>q$. Let $\mu$ be a positive Borel measure on $[0,1)$ which satisfies the condition in Theorem \ref{Th3.1}.
\begin{itemize}
\item [($\romannumeral1$)] If $\mathcal{H}_{\mu,\beta}$ is a compact operator from $\mathcal{D}_\alpha^p$ into $\mathcal{D}_{\beta-1}^q$ then $\mu$ is a vanishing $\left(\frac{\alpha+2}{p}+\frac{\beta+1}{q'}-1\right)$-Carleson measure.
\item [($\romannumeral2$)] If $\mu$ is a vanishing $\left(\max\{\alpha+2-p,\beta+1\}(\frac{1}{p}+\frac{1}{q'})\right)$-Carleson measure then $\mathcal{H}_{\mu,\beta}$ is a compact operator from $\mathcal{D}_\alpha^p$ into $\mathcal{D}_{\beta-1}^q$.
\end{itemize}
\end{theorem}

\begin{proof}
The proof is the same as that of Theorem \ref{Th4.3}(\romannumeral1) and Theorem \ref{Th5.1}(\romannumeral1). We omit the details here.
\end{proof}

\begin{corollary}
Suppose that $1 < p < \infty $, $\alpha>p-1$. Let $\mu$ be a positive Borel measure on $[0,1)$ which satisfies the condition in Theorem \ref{Th3.1}. Then $\mathcal{H}_{\mu,\alpha+1-p}$ is a compact operator from $\mathcal{D}_\alpha^p$ into $\mathcal{D}_{\alpha-p}^p$ if and only if $\mu$ is a vanishing $(\alpha+2-p)$-Carleson measure.
\end{corollary}

\begin{corollary}
Suppose that $q > 1$ and $0<p\leq q <\infty$, $\alpha>-1$ and $\beta>0$. Let $\mu$ be a positive Borel measure on $[0,1)$ which satisfies the condition in Theorem \ref{Th3.1}.
\begin{itemize}
\item [($\romannumeral1$)] If $\mathcal{H}_{\mu,\beta+q}$ is a compact operator from $A_\alpha^p$ into $A_{\beta-1}^q$ then $\mu$ is a vanishing $\left(\frac{\alpha+2}{p}+\frac{\beta+1+q}{q'}\right)$-Carleson measure.
\item [($\romannumeral2$)] If $\mu$ is a vanishing $\left(\max\{\alpha+2,\beta+1+q\}(\frac{1}{p}+\frac{1}{q'})\right)$-Carleson measure then $\mathcal{H}_{\mu,\beta+q}$ is a compact operator from $A_\alpha^p$ into $A_{\beta-1}^q$.
\end{itemize}
\end{corollary}

\begin{corollary}
Suppose that $1 < p < \infty $, $\alpha>p-1$. Let $\mu$ be a positive Borel measure on $[0,1)$ which satisfies the condition in Theorem \ref{Th3.1}. Then $\mathcal{H}_{\mu,\alpha+1}$ is a compact operator from $A_\alpha^p$ into $A_{\alpha-p}^p$ if and only if $\mu$ is a vanishing $(\alpha+2)$-Carleson measure.
\end{corollary}
%\subsection*{Conflicts of Interest}
%The authors declare that there are no conflicts of interest regarding the publication of this paper.
%\subsection*{Availability of data and material}

%The authors declare that all data and material in this paper are available.


\begin{thebibliography}{99}

\bibitem{BHM} Bao G, Wulan H. Hankel matrices acting on Dirichlet spaces, J  Math Anal Appl, 2014 \textbf{409}(1): 228-235.
\bibitem{DBS} Duren P L, Schuster A. Bergman spaces. American Mathematical Soc, 2004.
\bibitem{ZOT}Zhu K. Operator theory in function spaces, American Mathematical Soc, 2007.
\bibitem{CGH} Chatzifountas  C, Girela D, Pel\'{a}ez J \'{A}.
 A generalized Hilbert matrix acting on Hardy spaces, J  Math Anal Appl, 2014, \textbf{413}(1): 154-168.
\bibitem{CCO} Cowen Jr C C, MacCluer B I. Composition Operators on Spaces of Analytic
    Functions. CRC Press, 1995.
\bibitem{DCO} Diamantopoulos  E , Siskakis  A G. Composition operators and the Hilbert matrix,
 Studia Math, 2000, \textbf{140}(2): 191-198.

\bibitem{GGH} Galanopoulos  P , Girela  D, Pel\'{a}ez, J  \'{A} , Siskakis  A  G. Generalized
    Hilbert operators, Ann  Acad Sci Fenn  Math, 2014, \textbf{39}: 231-258.
\bibitem{GHM} Galanopoulos P , Pel\'{a}ez J \'{A}. A Hankel matrix acting on Hardy and Bergman
    spaces, Studia Math ,2010, \textbf{200}(3): 201-220.
\bibitem{GGHO} Girela  D , Merch\'{a}n  N. A generalized Hilbert operator acting on conformally
    invariant spaces, Banach J Math  Anal ,2018, \textbf{12}(2): 374-398.
    \bibitem{GHMA} Girela D, Merch\'{a}n N. Hankel matrices acting on the Hardy space $ H^ 1$  and on Dirichlet spaces, Revista Matem\'{a}tica Complutense, 2019, 32(3): 799-822.
\bibitem{GOB} Gnuschke-Hauschild D , Pommerenke C. On Bloch functions and gap series, J  Reine
    Angew Math, 1986, \textbf{367}: 172-186.
   \bibitem{JTC}Jevti\'{c} M, Vukoti\'{c} D, Arsenovi\'{c} M. Taylor coefficients and coefficient multipliers of Hardy and Bergman-type spaces, Springer International Publishing, 2016.
\bibitem{HCM} Hastings W W. A Carleson measure theorem for Bergman spaces, Proc Amer  Math
    Soc, 1975, \textbf{52}(1): 237-241.
\bibitem{LFR} Luecking D H, Forward and reverse Carleson inequalities for functions in Bergman spaces and their derivatives, American Journal of Mathematics, 1985, \textbf{107}(1): 85-111.
\bibitem{MVL} MacCluer B, Zhao R. Vanishing logarithmic Carleson measures, Illinois J  Math ,
    2002, \textbf{46}(2):507-518.
\bibitem{POB} Pommerenke C, Clunie J, Anderson J. On Bloch functions and normal functions, J
    Reine  Angew Math, 1974, \textbf{270}: 12-37.
\bibitem{ROM} Romberg B W, Duren P L, Shields A L. Linear functionals on Hp spaces with $0< p<
    1$,
  Reine Angew Math, 1969, \textbf{238}: 32-60.
\bibitem{YDBL}Ye S, Zhou Z. A derivative-Hilbert operator acting on the Bloch space,
  Complex Anal Oper Theory, 2021, \textbf{15}(5): 88.
\bibitem{YDBE}Ye S, Zhou Z. A derivative-Hilbert operator acting on Bergman spaces,
J  Math  Anal  Appl, 2022, \textbf{506}(1): 125553.
\bibitem{YGH} Ye S, Zhou, Z. Generalized Hilbert Operator Acting on Bloch Type Spaces, Acta Mathematic Sinica, Chinese Series (In Press), 2022.
\bibitem{ZOL} Zhao  R. On logarithmic Carleson measures, Acta Sci  Math, 2003, \textbf{69}:
    605-618.
\bibitem{ZBT} Zhu  K. Bloch Type Spaces of Analytic Functions, Rocky Mountain J Math, 1993,
    \textbf{23}(3): 1143-1177.



\end{thebibliography}
\end{document}